\documentclass[a4paper,12pt]{amsart}
\usepackage{amsrefs,amssymb,amsmath,amsthm,amsfonts,epsfig,graphicx,psfrag}
\usepackage[dvips]{color}
\theoremstyle{plain}
\newtheorem{teo}{Theorem}[section]
\newtheorem{cor}[teo]{Corollary}
\newtheorem{lemma}[teo]{Lemma}
\newtheorem{prop}[teo]{Proposition}

\theoremstyle{definition}
\newtheorem{df}[teo]{Definition}

\theoremstyle{remark}
\newtheorem{obs}[teo]{Remark}
\newtheorem{ejp}[teo]{Example}

\DeclareMathOperator{\diam}{diam}
\DeclareMathOperator{\clos}{clos}

\DeclareMathOperator{\dist}{dist} 

\newcommand{\reparam}{\mathcal{H}_0^+(\R)}

\newcommand{\fix}{\hbox{Sing}}
\newcommand{\sing}{\fix}

\newcommand{\R}    {\mathbb R}
\newcommand{\Q}    {\mathbb Q}

\newcommand{\Z}  {\mathbb Z}
\newcommand{\N}  {\mathbb N}
\renewcommand{\epsilon}{\varepsilon}

\begin{document}

\author[A. Artigue]{Alfonso Artigue}
\address{Regional Norte, Universidad de la Rep\'ublica, Gral. Rivera 1530, Salto, Uruguay}
\email{aartigue@fing.edu.uy}
\title[Expansive Flows of Surfaces]{Expansive Flows of Surfaces}
\date{\today}

\begin{abstract}
We prove that a flow on a compact surface is expansive if and only if the singularities are of saddle type and the union of their separatrices is dense.
Moreover we show that such flows are obtained by surgery on the suspension of minimal interval exchange maps.
\end{abstract}
\maketitle

\section*{Introduction}

\par Expansive homeomorphisms on surfaces are known to be conjugated to pseudo-Anosov maps \cites{H,L}. 
In particular the stable leaves form a minimal measured foliation.
Here we study flows on compact surfaces that are expansive in the sense of \cite{K}. We show a strong relation between
the orbit structure of expansive flows and the stable foliation of a pseudo-Anosov maps. 
In the lack of saddle connections we prove that expansive flows are suspensions of minimal interval exchange maps. Therefore
the orbits of said flows
make a minimal measured foliation.

\par The only published work on
the subject of expansive flows of surfaces, known to the author, is \cite{LG}. In our context their result means
that expansive flows on surfaces must present singular (i.e. equilibrium)
points. 

\par The main ideas of the present article arise from relating expansive dynamics,
polygonal billiards and singular flows. The expansive properties of polygonal
billiards were established in \cite{GKT}. In this article it is shown that excluding
periodic orbits the collision map is expansive and two close non-periodic points are separated
after visiting different sides of the polygon. The link between
flows with singularities and polygonal billiards is given in \cite{ZK}. 
In said work a flow on a compact surface is constructed from any given rational polygonal billiard and fixed direction.
The equilibrium points of the flow are singularities of 
saddle type associated with the corners of the polygon.

\par Our main result is Theorem \ref{charexpsup}. It states that a flow on a
compact surface, in the absence of singularities of saddle type of index 0, 
is expansive if and only if the singularities
are of saddle type and the union of their
separatrices is dense in the surface. Roughly speaking the proof is the following.
Take two points in different local orbits, 
consider a separatrix between them and follow them until the end of the separatrix. The separation 
occurs because the singularity is of saddle type of negative index. To prove the converse the key step is to show that
there are no wandering points. The rest of the argument is quite standard. 

\subsection*{Description of the content.}
\par Section 1. We introduce the definition of expansive flow and show 
the equivalence with the definition of $k^*$-expansive given in \cite{K}.
\par Section 2. We prove that expansive flows on surfaces do not have wandering points.
\par Section 3. We show that the singular points of expansive flows are of saddle type. 
Also we show how this kind of singularities helps to the expansive properties of the flow.
\par Section 4. We show that expansive flows are Cherry flows (in the sense of \cite{Gutierre}) 
and as a consequence we conclude that expansive flows are smoothable.
Previously we prove that expansive flows on surfaces do not have periodic orbits and have singular points, moreover
the union of their separatrices is dense in the surface.
\par Section 5. 
We show that expansiveness is invariant under adding and removing singularities of index 0.
\par Section 6. 
We give characterizations of expansive flows on surfaces. In the absence
of singularities of index 0 it is shown that expansiveness is equivalent with
(1) the singularities are of saddle type and the union of their separatrices is dense and
(2) there is a finite and positive number of singularities, $\Omega(\phi)=S$ and there are no periodic orbits.
Also we show that the surfaces that admits expansive flows are: 
the torus with $b$ boundaries, $h$ handles and $c$ cross-cups with $b+h+c>0$.
In particular the torus do not admit expansive flows.
\par Section 7. 
We show that the suspension of an interval exchange map $f$ is an 
expansive flow on a surface $S$ if and only if $S$ is not the torus and $f$ has not periodic points.
\par Section 8. In the discrete case it is known that the interval and the circle do not admit
expansive homeomorphisms (see \cite{JU}). Nevertheless it
is easy to prove that expansive flows on those spaces are the ones 
with a finite number of singular points. 
In particular this shows that expansive flows can posses saddle connections. 
In fact, Example \ref{bitoroexp} shows that there exist  
expansive flows with cycles of saddle connections that disconnect the surface. 
In Theorem \ref{structure} we give a procedure to construct 
every expansive flow on a compact surface, the procedure is: 
suspend some minimal interval exchange maps, add singularities of index 0, make cuts along saddle connections and glue them.
This procedure will give and expansive flow if and only if no connected component
of the surface obtained is the torus.
\par Section 9. As an application and a source of examples of expansive flows on surfaces we consider rational polygonal billiards.
As explained in \cite{ZK} such billiard flows can be seen on a compact surface and fixing a direction $v\in\R^2$ a 
flow is determined presenting singularities of saddle type in the corners. 
Theorem \ref{billar} states that the the associated flow is expansive if and only if 
there are no periodic orbits with 
initial direction $v$ in the billiard.

\section{Expansive flows}
In this section we introduce the definition of expansive flow.
Let $(X,\dist)$ be a compact metric space and
$\phi\colon \R\times X\to X$ be a continuous flow. 
A point $p\in X$ is said to be a \emph{singularity} of $\phi$ if $\phi_t(p)=p$ for all $t\in\R$. In other case it is called \emph{regular}.
The set of singularities is denoted by $\sing$.

\par It is easy to show that if $x$ is a regular point then 
for all $\delta>0$ there exist $s>0$ such that $y=\phi_s(x)\neq x$ and
$\dist(\phi_t(x),\phi_t(y))<\delta$ for all $t\in\R$.
Roughly speaking it means that 
two points in a small orbit segment contradict the expansiveness of the flow. So every definition of \emph{expansive flow} 
must consider this fact.
In \cite{BW} it 
is done in such a way that singular points must be isolated points of 
the space, therefore no manifold of positive dimension admits an expansive flows (in the sense of \cite{BW}) with singular points. 
In \cite{K} a definition is given in order to study expansive properties of the Lorenz attractor.
They used the term $k^*$-\emph{expansive}. 
Of course it is weaker than the definition in \cite{BW}, but just because it 
allows singular points. Our definition of expansive flow is equivalent with \cite{K} (as we show 
in Theorem \ref{kequiv}) but seems to be more insightful. 

\par First we define
another distance in $X$ that allow us to say in a formal way that two points are locally in the same orbit. Let
$$\dist_\phi(x,y)=\inf\{\diam(\phi_{[a,b]}(z)):z\in X, [a,b]\subset\R, x,y\in\phi_{[a,b]}(z)\}$$
if $y\in\phi_\R(x)$ and $\dist_\phi(x,y)=
\diam(X)$ if $y\notin\phi_\R(x)$.
Consider
$$\beta_0=\inf\{\diam(\phi_\R(x)):x\notin\sing\}.$$ 
Notice that $\dist_\phi(x,y)<\beta_0$ if and only if 
$x$ and $y$ are in an orbit segment of diameter less than $\beta_0$.
Let $\reparam$ be the set of increasing homeomorphisms $h\colon\R\to\R$
such that $h(0)=0$.

\begin{df}\label{expflow}
We say that $\phi$ is \emph{expansive} if for all $\beta>0$ there
exists an \emph{expansive constant} $\delta>0$ such that if
$\dist(\phi_{h(t)}(x),\phi_t(y))<\delta$ for all $t\in\R$ and some
$h\in\reparam$ then $\dist_\phi(x,y)<\beta$.
\end{df}

\par If an expansive flow presents a singularity $p$ that is not an isolated point of $X$, it is easy to see that
there exist $t_n\to\infty$ and $x_n\to p$, $x_n\neq p$, such that 
$\dist_\phi(x_n,\phi_{t_n}(x_n))\to 0$. That simple remark is the essential difference between our definition of
expansive flow and the one given in \cite{BW}.
\par In \cite{BW} (Theorem 3) they give four equivalent definitions of expansive flows. In fact it is not necessary
to assume that the flow do not present singular points in this Theorem, that is because each item implies that
the set of singular points is an isolated finite set (see \cite{Oka}). 
Our Definition \ref{expflow} 
above is the combination of items (ii) and (iii) 
of this Theorem 
suggested in \cite{BW} after its proof.
It may not be equivalent
in the presence of equilibrium points. In fact they are equivalent if and only if $\sing$ is an isolated set of the space.
In Theorem \ref{kequiv} we show that Definition \ref{expflow} coincides with the notion of $k^*$-\emph{expansive}
introduced in \cite{K}.

\begin{lemma}\label{lemakequiv}
Suppose that $\phi$ presents a finite number of singular points and
$\beta_0>0$. Then for all $\beta\in (0,\beta_0)$ there exist
$\delta>0$ such that if $\dist(\phi_{g(t)}(x),\phi_t(x))<\delta$
for all $t\in \R$ and a continuous function $g\colon\R\to\R$,
$g(0)=0$, then $\dist_\phi(\phi_{g(t)}(x),\phi_t(x)<\beta$ for all
$t\in\R$.
\end{lemma}

\begin{proof}
By contradiction suppose that there exist $\beta\in(0,\beta_0)$,
$\delta_n\to 0$, $x_n$ and continuous functions $g_n\colon\R\to\R$ such that
\begin{equation}\label{ecuexp}
    \dist(\phi_{g_n(t)}(x_n),\phi_t(x_n))<\delta_n
\end{equation}
for all $t\in\R$, for all $n\in\N$, $g_n(0)=0$ and also suppose that there exists $t'_n$ such that
$$\dist_\phi(\phi_{g_n(t'_n)}(x_n),\phi_{t'_n}(x_n))>\beta$$. Since
$\phi_{g_n(0)}(x_n)=\phi_{0}(x_n)$ there exist $t_n\in(0,t'_n)$
such that
$$\dist_\phi(\phi_{g_n(t_n)}(x_n),\phi_{t_n}(x_n))=\beta$$
Let $a_n=\phi_{g_n(t_n)}(x_n)$ and $b_n=\phi_{t_n}(x_n)$. By
equation (\ref{ecuexp}) we have that $\dist(a_n,b_n)<\delta_n\to
0$. Then we can suppose that $a_n,b_n\to c$. Also assume that $\phi_{s_n}(a_n)=b_n$ with $s_n>0$ and
$\diam(\phi_{[0,s_n]}(a_n))=\beta$. If $\{s_n\}_{n\in\N}$ is
bounded it is easy to prove that $c$ is a periodic point and
$\diam\phi_\R(c)=\beta$, contradicting
$\beta<\beta_0$. If $\{s_n\}_{n\in\N}$ is not bounded, eventually taking a subsequence, we have 
that $s_n\to +\infty$. Also we can suppose that
$\phi_{[0,s_n]}(a_n)$ converges to a compact set $K\subset X$ considering the
Hausdorff distance\footnote{If $A,B\subset X$ are compact sets the
Hausdorff distance between $A$ and $B$ is
$d_H(A,B)=\max\{\sup_{a\in A}\inf_{b\in B} \dist(a,b),\sup_{b\in
B}\inf_{a\in A} \dist(a,b)\}$. } between compact subsets of
$X$. Since $a_n,\phi_{s_n}(a_n)\to c$, it is easy to see that $K$ is an invariant set.
Also we have that $\diam(K)=\beta$ and $K$ is connected, therefore $K$ is an infinite set.
By hypothesis, the flow
presents a finite number of singularities and then there is a regular point in
$K$ whose orbit has diameter less or equal than $\beta$,
again the contradiction is that $\beta<\beta_0$.
\end{proof}

\begin{teo}\label{kequiv}
The following statements are equivalent:
\begin{enumerate}
 \item $\phi$ is expansive,
\item $\phi$ is $k^*$-expansive, i.e. for all $\epsilon>0$
there exist an expansive constant $\delta>0$ such that if
$\dist(\phi_{h(t)}(x),\phi_t(y))<\delta$ for all $t\in\R$ and some
$h\in\reparam$ then there exist $s,t_0\in\R$ such that
$|s|<\epsilon$ and $\phi_{h(t_0)}(x)=\phi_{t_0+s}(y)$.

\end{enumerate}
  
\end{teo}

\begin{proof}

Notice that in both cases the number of singular points is finite and
$\beta_0>0$ if $X$ is not a finite set. If $X$ is a finite set the
proof is trivial.

\par ($1\Rightarrow 2$) Expansiveness easily implies that if $\gamma>0$ is smaller than an expansive constant then for
all $x\notin \sing$ there exists $t\in\R$ such that $\phi_t(x)\notin
B_\gamma(\fix)$.
\par Fix $\epsilon>0$. Now we will show that there exist $\beta>0$ such that if
$\dist(x,\fix)\geq \gamma$ and $\diam(\phi_{[0,s]}(x))<\beta$ then
$|s|<\epsilon$. By contradiction, suppose that there exist $x_n\notin
B_\gamma(\fix)$ and $t_n>\epsilon$ such that
$\diam(\phi_{[0,t_n]}(x_n))\to 0$. Eventually taking a subsequence we can
suppose that $x_n\to z$. Notice that $z$ is a regular point and then there exists $\epsilon'\in(0,\epsilon)$
such that $\phi_{\epsilon'}(z)\neq z$. By the continuity of $\phi$
we have that $\phi_{\epsilon'}(x_n)$ converges to
$\phi_{\epsilon'}(z)$. This is a contradiction because 
$\diam(\phi_{[0,t_n]}(x_n))\to 0$.

\par For a value $\beta$
as in the previous paragraph there exist an expansive
constant $\delta$, by hypothesis. We will show that $\delta$ is an expansive constant for the value
$\epsilon$ fixed before. Suppose that
$\dist(\phi_{h(t)}(y),\phi_t(x))<\delta$ for all $t\in\R$, some
$x,y\in X$ and $h\in\reparam$. Without loss of generality we can suppose that $x\notin\fix$. Let
$t_0\in\R$ be such that $\phi_{t_0}(x)\notin B_\gamma(\fix)$. 
If we consider $h'\in\reparam$ given by $h'(t)=h(t+t_0)-h(t_0)$ we have that
\begin{gather*}
\dist(\phi_{h'(t)}(\phi_{h(t_0)}(y)),
\phi_t(\phi_{t_0}(x)))=\\
\dist(\phi_{h(t+t_0)-h(t_0)}(\phi_{h(t_0)}(y)),\phi_t(\phi_{t_0}(x)))=\\
\dist(\phi_{h(t+t_0)}(y),\phi_{t+t_0}(x))<\delta
\end{gather*}
for all $t\in \R$. We conclude by hypothesis that
$\dist_\phi(\phi_{h(t_0)}(y),\phi_{t_0}(x))<\beta$. Since
$\phi_{t_0}(x)\notin B_\gamma(\fix)$, there exists
$s\in(-\epsilon,\epsilon)$ such that
$\phi_{h(t_0)}(y)=\phi_{t_0+s}(x)$ and the proof ends.

\par ($2\Rightarrow 1$) Given any $\beta>0$, without loss of generality we suppose $\beta\in
(0,\beta_0)$. We apply Lemma \ref{lemakequiv} to some value
$\beta'\in(0,\beta)$ and we have the associated value $\delta'>0$.
Let $\delta''\in(0,\delta')$. By the continuity of the flow there
exists $\epsilon>0$ such that
\begin{equation}\label{epsilon1}
\hbox{if } \diam(\phi_{[a,b]}(x))<\beta' \hbox{ then }
\diam(\phi_{[a-\epsilon,b+\epsilon]}(x))<\beta
\end{equation}
Assuming that $\epsilon$ and $\delta''$ are sufficiently small we can suppose that
\begin{equation}\label{epsilon2}
\hbox{if }\dist(x,y)<\delta''\hbox{ then
}\dist(\phi_s(x),y)<\delta' \hbox{ for all }s\in(-\epsilon,\epsilon)
\end{equation}
By hypothesis there exists an expansive constant $\delta'''$
associated with $\epsilon$. We will show that any positive value
$\delta<\min\{\delta',\delta'',\delta'''\}$ is an expansive constant for the value $\beta$ fixed before. Suppose that
\begin{equation}\label{ecuexp2}
    \dist(\phi_{h(t)}(x),\phi_t(y))<\delta,\hbox{ for all }t\in\R \hbox{ and some } h\in\reparam
\end{equation}
Then by hypothesis there exist $s,t_0\in\R$ such that
$\phi_{h(t_0)}(x)=\phi_{t_0+s}(y)$ and $|s|<\epsilon$. Define
$z=\phi_{h(t_0)}(x)$ and $g(t)=h(t_0+t)-h(t_0)$. By
(\ref{ecuexp2}) we have that
$\dist(\phi_{g(t)}(z),\phi_{t-s}(z))<\delta$ for all $t\in \R$,
that is because
\[
\begin{array}{l}
\phi_{g(t)}(z)=\phi_{h(t_0+t)-h(t_0)}(\phi_{h(t_0)}(x))=\phi_{h(t_0+t)}(x) \hbox{ and }\\
\phi_{t-s}(z)= \phi_{t-s}(\phi_{h(t_0)}(x))=\phi_{t+t_0}(\phi_{h(t_0)-s-t_0}(x))=\phi_{t+t_0}(y)
\end{array}
\]
Since $|s|<\epsilon$ if we consider (\ref{epsilon2}) we have that
$\dist(\phi_{g(t)}(z),\phi_t(z))<\delta'$ for all $t\in\R$, because
$\delta<\delta''$. Then applying Lemma \ref{lemakequiv} we have that
$$\dist_\phi(\phi_{g(-t_0)}(z),\phi_{-t_0}(z))<\beta'$$ This
points are $x$ and $\phi_s(y)$ respectively. Finally, by condition (\ref{epsilon1})
 we have that
$\dist_\phi(x,y)<\beta$.
\end{proof}

\begin{obs}\label{coroexpkomuro}
 Notice that in the previous proof (2$\Rightarrow$1) we have shown the following result: 
if $\phi$ has a finite number of singularities and $\beta_0>0$ then for all $\beta>0$ there
exist $\epsilon,\delta>0$ such that if $\dist(\phi_{h(t)}(x),\phi_t(y))<\delta$ for all $t\in\R$ with $x,y\in X$, 
 $h\in\reparam$ and $\phi_{h(t_0)}(x)=\phi_{t_0+s}(y)$ for some $t_0\in\R$ and $s\in(-\epsilon,\epsilon)$ then $\dist_\phi(x,y)<\beta$.
%Moreover we can conclude that $\dist_\phi(\phi_{h(t)}(x),\phi_t(y))<\delta$ for all $t\in\R$, that can be done
%just replacing $x$ by $\phi_{h(t)}(x)$ and $y$ by $\phi_t(y)$ and apply the same result for each $t\in\R$.
\end{obs}

\par The equivalence of the definitions of expansive flow considered in \cite{BW} 
and $k^*-$expansive given in \cite{K}, in the lack of singular points, was shown in \cite{Oka}.

\section{Wandering points}

In this section we shall prove that expansive flows on surfaces do not present wandering points. 
We consider a continuous flow $\phi$ acting on a compact surface $S$.
First we will recall some basic tools.

\par Consider an embedded segment $l\subset S$. We say that $l$ is a \emph{local cross section} 
of time $\tau>0$ for the flow if $\phi$ maps $[-\tau,\tau]\times l$ homeomorphically onto 
$\phi_{[-\tau,\tau]}(l)$. If $a,b\colon l\to(0,\tau)$ are continuous then the set 
$U=\{\phi_t(x): x\in l\hbox{ and } t\in (-a(x),b(x))\}$ is called a \emph{flow box}. 
In \cite{W} it is shown that every regular point belongs to a local cross section. 

\begin{lemma}\label{cubrimiento} If the flow $\phi$ presents a finite number of singularities
then $S\setminus \sing =\cup_{i=1}^{+\infty}U_i$ where:
\begin{itemize}
\item each $U_i$ is a flow box,
\item each compact subset of $S\setminus\sing$ is contained in finitely many flow boxes $U_i$,
\item if $i\neq j$ then $U_i\cap U_j\subset\partial U_i\cap \partial U_j$
\end{itemize}
\end{lemma}

\begin{proof}
 See Proposition 4.3 of \cite{Gutierrez}.
\end{proof}

As usual we define the $\omega$-limit set of a point
$x$ as the $\omega(x)=\{a\in S:\exists
t_n\to+\infty/\phi_{t_n}(x)\to a\}$.

\begin{lemma}\label{peixoto2}

Let $l=[a,b]$ and $l'$ be two compact local cross sections and
$\tau\colon [a,b)\to\R$ be a continuous function such that
$\phi_{\tau(x)}(x)\in l'$ for all $x\in[a,b)$ and $\lim_{x\to b}
\tau(x)=+\infty$. Then $\omega(b)\subset\sing$.
\end{lemma}
\begin{proof}
By contradiction suppose that there exist a regular point $y\in\omega(b)$. 
We can assume that $y\notin l\cup l'$. 
Consider a compact local cross section $j$ such that $y\in j$ and since $l$ and $l'$ are compact we can suppose that 
$j\cap l=j\cap l'=\emptyset$.
For all $x\in [a,b)$ consider the set $T_x=\{t\in
[0,\tau(x)]:\phi_t(x)\in j\}$ and define $N(x)\in\Z$ as
the number of points in $T_x$.
The continuity of $\phi$ implies that $N$ is continuous at $x$ if 
the points in $\phi_{T_x}(x)$ are not in the boundary of the interval $j$.
Therefore $N$ has at most two points of discontinuity and then $N$ is bounded.
On the other hand there is an infinite number of values of $t>0$ such that 
$\phi_t(b)\in
j$ because $y\in\omega(b)$. 
We also have that $\tau(x)\to +\infty$ as $x\to b$, thus $\lim_{x\to b} N(x)=\infty$, which is an absurd.
\end{proof}

 If $p$ is a singularity
and $x$ is a regular point such that $\phi_t(x)\to p$ as $t\to +\infty$ (resp. $t\to -\infty$)
then the orbit of $x$ is said to be a \emph{stable} (resp. \emph{unstable}) \emph{separatrix} of $p$.
We say that a point $x\in S$ is \emph{stable} if for all
$\epsilon>0$ there exist $\delta>0$ such that if $y\in B_\delta (x)$
then $\dist(\phi_t(y),\phi_t(x))<\epsilon$ for all $t>0$. A point $x\in S$
is \emph{asymptotically stable} if it is stable and there exists $r>0$ such that if $y\in B_r(x)$ then 
$\dist(\phi_t(y),\phi_t(x))\to 0$ 
as $t\to +\infty$.

\begin{lemma}\label{muchassep}
If a singularity of $\phi$
presents an infinite number of separatrices then at least
one of them is asymptotically stable.
\end{lemma}

\begin{proof}
It follows by the arguments in \cite{Hartman} (pag. 161).
\end{proof}

\par We say that $x\in S$ is a \emph{wandering} point if there exist a neighborhood $U$ of $x$ and
 $\tau>0$ such that $\phi_t(U)\cap
 U=\emptyset$ for all $t>\tau$. We denote by $\Omega(\phi)$ the set of non-wandering
 points.

\begin{prop}\label{wanderings}
 If $\phi$ is expansive then $\Omega(\phi)=S$.
\end{prop}

\begin{proof}
By contradiction suppose that there are wandering points. 
Then it is easy to see that there exists a local cross section $l$ such that
$\phi_t(l)\cap l=\emptyset$ for all $t\neq 0$. We will show
that there exists a subsegment $l'\subset l$ that contradicts the
expansiveness for positive values of $t$ and then arguing the same
for the opposite flow, we arrive to a contradiction.
Fix an expansive constant $\delta>0$. We study two possible cases.

\par Case 1. Suppose that there is an infinite number of points in stable separatrices in $l$. 
Since there is a finite number of singularities, there exists $p\in\sing$ with infinitely many separatrices meeting $l$. 
By Lemma \ref{muchassep} one
of them is asymptotically stable. Take $x\in l$ in a
asymptotically stable separatrix. Therefore there exist $\mu>0$
such that if $\dist(x,y)<\mu$ then
$\dist(\phi_t(x),\phi_t(y)<\delta/2$ for all $t\geq 0$. Finally take
$l'\subset l\cap B_\mu (x)$. For all $y,z\in l'$ it holds that $\dist(\phi_t(y),\phi_t(z))<\delta$ for all $t\geq 0$.

\par Case 2. Suppose that there is a finite number of stable separatrices meeting $l$. Then there exist
$l'\subset l$ such that for all $x\in l'$ it holds $\omega(x)$ is
not a singular point. In this case we say that $l'$ satisfies condition (1). For each singularity $p$ consider a disc $D_p$ around $p$ of diameter
less than $\delta/2$. Take
the covering $\{U_i:i\in\N\}$ of $S\setminus \sing$ given by Lemma \ref{cubrimiento}.
Making a subdivision of each $U_i$ we can suppose that $\diam(U_i)<\delta/2$ for all $i\in\N$.
Reordering the flow boxes we can suppose that the sets $D_p$ and $U_1,\dots,U_N$ make a finite covering of the surface. 
Considering $l'$ smaller we can also
suppose that orbits of the orbit segments in $\partial U_i$ do not meet $l'$, call it condition (2). 
Let $a_i$ be the local cross section of $\partial U_i$ where the flow enters to the flow box. We can also suppose that
$l'$ do not meet any $a_i$.
\par Fix $x,y\in l'$ and define $A=\cup_{i=1}^N a_i$. Hence there exist two
increasing and divergent sequences $t_n,s_n\in \R^+$ such that
$\{t_n:n\in\N\}=\{t\in\R^+:\phi_t(x)\in A\}$ and
$\{s_n:n\in\N\}=\{t\in\R^+:\phi_t(y)\in A\}$. Let
$I,J\colon\N\to\{1,\dots,N\}$ the functions given by:
$\phi_{t_n}(x)\in a_{I(n)}$ and $\phi_{s_n}(y)\in a_{J(n)}$.

\par By induction we will show that $I=J$. \emph{Base case}.
Let $l''=[x,y]\subset l'$ be the segment determined by the points
$x$ and $y$ and
$$X=\{z\in l'':\exists t>0/\phi_{[0,t]}(z)\cap a_{J(1)}=\emptyset
\hbox{ and } \phi_t(z)\in a_{I(1)}\}$$ We will prove that $y\in
X$. We have that $x\in X$. Also $X$ is an open set in $l''$ by
condition (2). Let $Y$ be the connected component of $X$ that
contains $x$. Then $Y$ is an interval. Let $u$ be the extreme of
$Y$, $u\neq x$. We will show that $u\in Y$ and therefore $u=y$ and
$y\in Y$. By the definition of $X$ we have that there exist a
continuous function $T\colon [x,u)\to \R^+$ such that
$\phi_{T(z)}(z)\in a_{I(1)}$ and $\phi_{[0,T(z)]}(z)\cap
a_{J(1)}=\emptyset$. By condition (1) we have that $T(z)$ can not
converge to $\infty$ as $z\to u$, because of Lemma \ref{peixoto2}.
Hence there exist $z_n\in [x,u)$ such that $z_n\to u$ and
$T(z_n)\to T_u$. Then $\phi_{T(z_n)}(z_n)\to \phi_{T_u}(u)$. On
the other side $\phi_{[0,T_u]}(u)$ do not meet $a_{J(1)}$, because
condition (2) implies that if $n$ is sufficiently big,
$\phi_{[0,T(z_n)]}(z_n)\cap a_{J(1)}\neq\emptyset$. Then we have
that $u\in Y$ and $I(1)=J(1)$. \emph{Inductive step}. In order to
repeat the previous argument note that if $I(k)=J(k)$ for all
$k=1,\dots,K$, if we define
$l_K=[\phi_{t_K}(x),\phi_{s_K(y)}]\subset a_{I(K)}$, then $l_K$
also verifies conditions (1) and (2).

\par Let $h\colon[0,+\infty)\to [0,+\infty)$ be such that $h(0)=0$,
$h(t_n)=s_n$ for all $n\in\N$ and extended linearly. In this way
$h$ in a homeomorphism. We will show that
$\dist(\phi_t(x),\phi_{h(t)}(y))<\delta$ for all $t\geq 0$. Fix
$n\in\N$ and let
$$t^*=\sup \{t\geq t_n:\phi_{[t_n,t]}(x)\subset U_i \hbox{ para
alg\'un } i=1,\dots N\}.$$ Let $i_0$ be such that
$\phi_{t^*}(x)\in b_{i_0}$. If $t^*\geq t_{n+1}$ then both
segments $\phi_{[t_n,t_{n+1}]}(x)$ and $\phi_{[s_n,s_{n+1}]}(y)$
are contained in $U_{i_0}$. Therefore
$$\dist(\phi_t(x),\phi_{h(t)}(y)<\delta$$ for all
$t\in[t_n,t_{n+1}]$, because $\diam(U_{i_0})<\delta/2$. Suppose
that $t^*<t_{n+1}$. Then $x^*=\phi_{t^*}(x)\in D_p$ for some
$p\in\sing$. Hence there exist $s^*\geq s_n$ such that
$y^*=\phi_{s^*}(y)\in b_{i_0}$ and $\phi_{[s_n,s^*)}(y)$ are
contained in $U_{i_0}$. Then $[x^*,y^*]\subset b_i$ is contained
in $D_p$. Therefore $\phi_{[t^*,t_{n+1}]}(x)$ and
$\phi_{[s^*,s_{n+1}]}(y)$ are contained in $D_p$. It only rest to
note that $\diam(D_p\cup U_{i_0})<\delta$.
\end{proof}

\section{Singularities of saddle type}

\par An isolated singularity is said to be of \emph{saddle type} if it has a positive and finite number of separatrices.

\begin{prop}
\label{singularities}
  If $\phi$ is expansive then its singularities are of saddle type.
\end{prop}

\begin{proof} Expansiveness easily implies that the singularities are isolated. 
It is easy to see that if $p\in\sing$ is stable
then $p$ is asymptotically stable. But there are no asymptotically stable singularity because by
Proposition \ref{wanderings} there are no wandering points. So $p$ is not stable and we can apply
the arguments in the proof of Lemma 8.1 in \cite{Hartman} to conclude that there exists at least one
separatrix associated to $p$. Again by Proposition \ref{wanderings} there are no wandering points, thus there can not be an
asymptotically stable separatrix. Then by Lemma
\ref{muchassep} there is a finite number of separatrices.
\end{proof}

Now we will show the local behavior of the flow near a singularity of saddle type $p\in\sing$.
Consider an embedded disc $D\subset S$ such that $\sing\cap \clos(D)=\{p\}$. If $D$ is small enough
we can suppose that there is no separatrix of $p$ contained in $\clos(D)$. Using the Poincar\'e-Bendixon Theorem
we have that if $x\in D$ and $\phi_{\R^+}(x)\subset D$ or $\phi_{\R^-}(x)\subset D$ then $x$ belongs to
a separatrix of $p$. Moreover if $\phi_\R(x)\subset D$ then $x=p$. Now consider the set 
$$U=\{x\in D: \phi_{\R^+}(x)\nsubseteq \clos (D)\hbox{ and } \phi_{\R^-}(x)\nsubseteq \clos (D)\}$$%D, \hbox{ there exist }t_1<0<t_2 \hbox{ such that } \phi_{t_1}(x),\phi_{t_2}(x)\notin \clos(D)\}.$$
It is easy to see that $U$ is an open set. 
The connected components of $U$ will be called \emph{hyperbolic sectors}. Since $p$ has a finite number of separatrices
it is easy to see that there is a finite number of hyperbolic sectors in $D$.
The \emph{index} of a singularity of saddle type is 1-$n_h/2$, being $n_h$ the number of hyperbolic sectors in $D_p$.
If $p\in\partial S$ similar considerations can be made and we define the index of $p$ as $1-n_h$.
 In \cite{Hartman} Theorem 9.1 it is shown that this definition coincides with the usual notion of index for singular points
not lying in the boundary of the surface.

\par As usual we define the \emph{stable} and the \emph{unstable set} of a singular point $p$ as 
$$W^s(p)=\{x\in S:\phi_t(x)\to p\hbox{ as } t\to+\infty\}$$ 
$$W^u(p)=\{x\in S:\phi_t(x)\to p\hbox{ as } t\to-\infty\}$$
respectively.
\begin{df}\label{neighborhoodsingular}
If $p\in\sing$ is of saddle type we say that an embedded disc (or half disc if $p\in\partial S$) 
$D_p$ is an \emph{adapted} neighborhood of $p$ if
$\overline D_p\cap \fix=\{p\}$ and $\partial
D_p=\cup_{i=1}^{i=n}(\alpha_i\cup\beta_i\cup\gamma_i^+\cup\gamma^-_i)$.
Where $\alpha_i$ and $\beta_i$ are orbit segments and
$\gamma_i^\pm$ are local cross sections,
    such that there exist $x_i\in \gamma_i^+$ and $y_i\in\gamma^-_i$
    such that $W^u(p)\setminus\{p\}=\bigcup_{i=0}^n \phi_\R(x_i)$,
    $W^s(p)\setminus\{p\}=\cup_{i=0}^n \phi_\R(y_i)$. See figure \ref{entorno}.
\end{df}

\begin{figure}%[htbp]
   \input{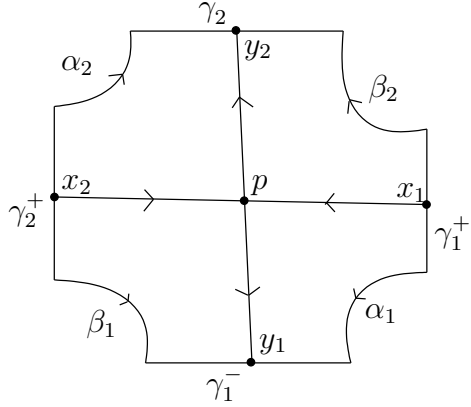}
    \caption{An adapted neighborhood of a singularity of index -1.}
   \label{entorno}
\end{figure}

\par Considering the analysis of hyperbolic sectors made in \cite{Hartman} (pag. 167) 
we have that every singularity of saddle type has an adapted neighborhood.
The following result shows how a singularity of negative index gives a kind of local expansiveness.

\begin{lemma}\label{sillaexp}
Suppose that $p$ is a singularity of saddle type of negative index and take an adapted neighborhood $D_p$ of $p$. Then
for all $\epsilon>0$ there exists $\delta>0$ such that if $x,y\in
D_p$ and $\dist(\phi_t(x),\phi_{h(t)}(y))<\delta$ for all $t\in\R$
and for some $h\in\reparam$, then $x$ and $y$ belongs to the same hyperbolic sector 
or there exist $t_0,s\in\R$ such that $\phi_{h(t_0)}(y)=\phi_{t_0+s}(x)$ and $|s|<\epsilon$.
\end{lemma}

\begin{proof}
We give the proof for the case $p\notin\partial S$, the other case is similar.
In the notation of the Definition \ref{neighborhoodsingular} define 
$W=\{p\}\cup_{i=1}^{n_h/2}\phi_{\R^+}(x_i)\cup_{i=1}^{n_h/2}\phi_{\R^-}(y_i)$.
It is the complement of the hyperbolic sectors in $D_p$.
Consider $\delta>0$ such that the following conditions hold:
\begin{enumerate}
 \item if $z\in W$ and
    $z\in B_\delta(x_i)$ or $z\in
    B_\delta(y_i)$ then there exist $s\in(-\epsilon,\epsilon)$
    such that $\phi_s(z)=x_i$ ($\phi_s(z)=y_i$) and
\item for all hyperbolic sector $U_j$ and $\gamma_i^\pm$
if $\dist(\gamma_i^\pm,U_j)<\delta$ then $\dist(\gamma_i^\pm,U_j)=0$.
\end{enumerate}
    \par Now we study three possible cases.
    \par Case 1: $x\in W$, $y\notin W$.  Since
        $y\notin W$ there exist $t'>0$ such that
        $\phi_{h(t')}(y)\in\gamma_i^+$ for some $i$. 
Without loss of generality
    suppose that $\lim_{t\to+\infty}\phi_t(x)=p$. Since the index of $p$ is negative, there exist a
hyperbolic sector $U_k$ such that $\phi_t(x)\in \partial U_k$ for all $t\geq 0$ and $\dist(U_k,\gamma_i^+)\geq \delta$. 
Therefore
        $\dist(\phi_{h(t')}(y),\phi_{t'}(x))\geq \delta_1$.
        \par Case 2: $x,y\in W$. If $x=y=p$ the case is trivial. If
        $x,y\in \phi_{\R^+}(x_i)$ (or $\phi_{\R^-}(y_i)$, the argument is the same) we conclude by condition (1). 
	If $x$ and $y$ are in different connected components of $W\setminus\{p\}$
	we conclude again by condition (2).
        \par Case 3: $x,y\notin W$ but not in the same hyperbolic sector. We conclude by condition (2).
\end{proof}

\section{Smooth models for expansive flows}

In this section we will show that continuous expansive flows on $C^\infty$ compact surfaces are topologically
equivalent to flows of $C^\infty$ class.
Two flows $\phi\colon \R\times S\to S$ and $\phi'\colon \R\times S'\to S'$ are \emph{topological equivalent} if there exists a homeomorphisms
from $S$ to $S'$ preserving the orbits.
\par A regular point $x\in S$ is said to be \emph{periodic} if there exists $t>0$ such that $x=\phi_t(x)$. We do not consider singularities
as periodic points.

\begin{prop}\label{expnoper}
Expansive flows on compact surfaces do not have periodic points.
\end{prop}

\begin{proof}
By contradiction suppose that $x\in S$ is a periodic point. Take a
transversal $l$ trough $x$ and consider the first return map
$f\colon l'\subset l\to l$ defined in a smaller section $l'$.
Expansive implies that the map $f$ can not have fixed points
close to $x$. Therefore there exist wandering points,
contradicting Proposition \ref{wanderings}.
\end{proof}

An embedded circle $\gamma\subset S$ is said to be a \emph{transversal circle} if for all
$x\in\gamma$ there exists a neighborhood $U$ of $x$ such that $U\cap\gamma$ is a local cross section.

\begin{lemma}\label{irrationaltorus}
Suppose that $\phi$ do not have periodic points, $\sing$ is a finite set, 
the union of the stable separatrices is not dense
in the surface and $\Omega(\phi)=S$. Then $S$ is the torus and $\phi$ is an irrational flow.
\end{lemma}

\begin{proof}
 Since the union of the stable separatrices is not dense and $\sing$ is a finite set there exists a cross section $l$
without points in separatrices. There exist $x\in l$ and $t>0$ such that $\phi_t(x)\in l$ because $\Omega(\phi)=S$.
As there are no periodic points we have that $\phi_t(x)\neq x$. Let $l^*=\{y\in l: \phi_{\R^+}(y)\cap l\neq\emptyset\}$
and consider $f\colon l^*\to l$ the first return map. 
If the surface is not orientable then $f$ may reverse orientation in some connected component of $l^*$. But applying
Lemma \ref{peixoto2} we can easily show that there exist a periodic point. So we can suppose that $f$ preserves orientation
and with standard techniques it can be shown that there exists a transversal circle $\gamma$. Moreover we can suppose 
that $\gamma\subset \phi_\R(l)$. Therefore there is no separatrix meeting $\gamma$. Now we consider $\gamma^*$ as 
the set of points of $\gamma$ that returns to $\gamma$. Since there are no points in separatrices in $\gamma$ we can
apply Lemma \ref{peixoto2} to conclude that $\gamma^*=\gamma$. So the first return map of $g\colon\gamma\to\gamma$ is an homeomorphism.
Since $\Omega(\phi)=S$ and there are no periodic points we have that $g$ is conjugated to an irrational rotation on the circle $\gamma$. Then 
$\phi$ is conjugated to a suspension of $g$ and $S$ is the torus.
\end{proof}

Singularities play an important role in the subject of expansive flows on surfaces. It is shown in the following result.

\begin{prop}\label{separatrices densas2}
 If $\phi$ is an expansive flow on a compact surface $S$ then the union of the stable 
separatrices is dense in S. In particular $\sing\neq\emptyset$.
\end{prop}
\begin{proof} Arguing by
contradiction, suppose that the union of the stable separatrices is not
dense in $S$. By Proposition \ref{expnoper} $\phi$ has no periodic points and by Proposition \ref{wanderings} we have
that $\Omega(\phi)=S$. So we can apply Lemma \ref{irrationaltorus} and conclude
that $\phi$ is a suspension of an irrational rotation of the circle. In particular there are no singularities and
then $\phi$ is expansive in the sense of \cite{BW}. But this is a contradiction because there are no expansive homeomorphisms
on the circle (Theorem 6 in \cite{BW} shows that a homeomorphism is expansive if and only if its suspensions are expansive flows).
 
\end{proof}

\par In \cite{Gutierre} it is shown that Cherry flows are topologically equivalent to $C^\infty$ flows. The definition of Cherry flows
given in \cite{Gutierre} is the following. A continuous flow $\phi\colon \R\times S\to S$ is a \emph{Cherry flow} if the following conditions hold.
\begin{enumerate}
 \item $\phi$ has only finitely many singular points.
\item Let $p_1, p_2, \dots, p_m$ be the source singular points of $\phi$ and let $\lambda_1,\lambda_2,\dots,\lambda_m$ 
be their basins of repulsion. Then each $\lambda_j$ contains a unique trajectory $\theta_j$ connecting
$p_j$ to another (unique) fixed point $q_j\in\partial \lambda_j$.
\item There are finitely many points $x_1,x_2,\dots,x_n$ such that 
$$(\cup_{i=1}^m \lambda_i)\cup(\cup_{j=1}^n\phi_{\R^+}(x_j))$$ is dense in $S$.
\end{enumerate}

\begin{prop}\label{Cherry}
 Expansive flows are Cherry flows.
\end{prop}

\begin{proof}
 Expansiveness 
easily implies that $\sing$ is a finite set. By Proposition \ref{wanderings} we have that there are no source singular points, 
so item 2 of the definition of Cherry flow need not to be verified. In order to check item 3, notice that Proposition 
\ref{singularities} implies that there is a finite number of 
stable separatrices and by 
Proposition \ref{separatrices densas2} their union is dense in $S$.
\end{proof}

\begin{teo}\label{smooth}
Expansive flows on compact surfaces are topologically equivalent to $C^\infty$ flows.
\end{teo}

\begin{proof}
 Just apply the results in \cite{Gutierre} and Proposition \ref{Cherry}.
\end{proof}

\section{Removable singularities}

In that section we deal with singularities of index zero. We show that they can be \emph{removed} or \emph{added} without
loss of expansiveness.

\begin{df}
\par Consider two flows $\phi$ and $\psi$ defined on the same surface $S$ and 
let $p\in\sing$ be a singular points of index 0 of $\phi$.
Suppose that:
\begin{enumerate}
\item $p$ is non-singular for $\psi$,
\item for all $x\notin\psi_\R(p)$ it holds
$\phi_\R(x)=\psi_\R(x)$ and
\item the direction of both flows coincide on each orbit.
\end{enumerate}
In that case we say that $\psi$ \emph{removes} the singularity $p$ of
$\phi$ or equivalently $\phi$ \emph{adds} a singularity to $\psi$. 
\end{df}

\par It is easy to see that every singular point of index 0 can be removed. Conversely, given any regular point $x\in S$
there exist a flow that adds a singularity of index 0 at $x$. In fact the only singular points that can be added or removed are those of saddle type of index 0.
That kind of singular points are also called \emph{impassable grains} or \emph{fake saddles} 
in the literature. 

\par On surfaces one can remove or add singularities without loosing expansiveness. 
The following proof 
works on compact metric spaces. The converse will be shown on surfaces in Theorem \ref{expsinsing2}

\begin{prop}\label{expsinsing} If $\psi$ is expansive and removes a singular point of $\phi$
then $\phi$ is expansive.
\end{prop}

\begin{proof}
Let $p\in X$ be the singular point that $\psi$ removes from
$\phi$. Consider a local cross section $l$ through $p$ and take $t^{*}>0$ such that $\psi|_{(-t^{*},t^{*})\times l}$ 
is a homeomorphism. Let $U=\psi_{(-t^*,t^{*})}(l)$ be a flow box.
For any $\beta>0$ consider an expansive constant $\delta>0$ of $\psi$
with $\delta<\frac12\dist(p,\partial U)$. Suppose that
$\dist(\phi_t(x),\phi_{h(t)}(y))<\delta$ for all $t\in\R$ for some
$h\in\reparam$. Now we study the only three possible cases.
\par Case 1. Suppose that $x,y\notin \psi_\R(p)$. In that case there exist $g_x,g_y\in\reparam$
such that $\phi_t(x)=\psi_{g_x(t)}(x)$ and
$\phi_t(y)=\psi_{g_y(t)}(y)$ for all $t\in\R$. Then $\dist(\psi_t(x),
\psi_{g_x^{-1}\circ h\circ g_y(t)}(y))<\delta$ for all $t\in\R$ being $g_x^{-1}\circ h\circ g_y\in\reparam$.
Therefore $\dist_\psi(x,y)<\beta$ by hypothesis. Since $x$ and $y$ are not in
$\psi_\R(p)$ we have that $\dist_\psi(x,y)=\dist_\phi(x,y)<\beta$.
\par Case 2. Consider $x\in \psi_\R(p)$ and $y\notin \psi_\R(p)$. Without loss of generality
we can suppose that $\phi_t(x)\to p$ as $t\to +\infty$. Notice that
for all $t'>0$ there exist $t>t'$ such that $\phi_t(y)\notin U$.
This contradicts that $\dist(\phi_t(x),\phi_{h(t)}(y))<\delta<\frac12\dist(p,\partial U)$ for all $t\in\R$.
\par Case 3. Assume that $x,y\in \psi_\R(p)$. It easily implies that $\dist_\psi(x,y)=\dist_\phi(x,y)<\beta$.
\end{proof}

\begin{lemma} \label{separatricesdensas}
Suppose that the singularities of $\phi$ are of saddle type. 
Then for all $\beta>0$ there exists $\delta>0$ such that if 
$\dist(\phi_t(x),\phi_{h(t)}(y))<\delta$ for all $t\in\R$ with $\dist_\phi(x,y)\geq\beta$ and $h\in\reparam$ 
then there exists a local cross section $l$ containing $x$ and $y$ such that every separatrix
meeting $l$ between $x$ and $y$ is associated to a singularity of index 0.
\end{lemma}

\begin{proof} Consider any value of $\beta>0$ given. Since singularities of saddle type are isolated by definition, 
$\sing$ is a finite set and $$\beta_0=\inf\{\diam(\phi_\R(x)):x\notin\sing\}>0.$$ So we can apply the result stated in Remark \ref{coroexpkomuro}
and we have the constants $\epsilon$ and $\delta''$ given there.
For each singular point $p$ take an adapted neighborhood $D_p$. Suppose that $\diam(D_p)<\beta$ for all $p\in\sing$.
For each
singular point $p$ take the value $\delta_p$ given by Lemma
\ref{sillaexp} (considering the value of $\epsilon$ already fixed).
Consider the flow boxes $U_i$ given by
Lemma \ref{cubrimiento}. Suppose that $\mathcal C$ is a finite covering of $S$ made with $D_p$, $p\in\sing$, 
and a finite number of flow boxes $U_i$, $i=1,\dots,N$.
Eventually subdividing the flow boxes we can assume that $\diam(U_i)<\beta$
for all $i=1,\dots,N$.
We can also suppose that the
intersection of any two open sets in $\mathcal C$ is connected or
empty. Also take $\delta'$ such that if the diameter of
$X\subset S$ is less than $\delta'$, then $X$ is contained in some
open set of $\mathcal C$. Finally define $\delta =\min\{\delta',\delta'',\delta_p\}_{p\in\sing}$.

\par To show that $\delta$ works
suppose that $$\dist(\phi_t(x),\phi_{h(t)}(y))<\delta$$ for all
$t\in\R$ and some $h\in\reparam$. 

\par First we will show that there exists a local cross section through $x$ and $y$. 
Since $\dist(x,y)<\delta\leq \delta'$
we have that $x,y\in U$ for some $U\in\mathcal C$. 
Since $\diam (U)<\beta$ and $\dist_\phi(x,y)\geq\beta$ we have that $x$ and $y$ are 
not in an orbit segment contained in $U$.
If $U$ is a flow box or an adapted neighborhood of a singular point of index 0 
the prove is trivial.
Suppose that $x,y\in D_p$ for some
$p\in\sing$ of negative index. Again, if $x$ and $y$ are in the same hyperbolic sector the proof is easy.
Supposing that this is not the case we will arrive to a contradiction. Applying Lemma \ref{sillaexp} we have that there exist $t_0,s\in\R$ such that $\phi_{h(t_0)}(y)=\phi_{t_0+s}(x)$
and $|s|<\epsilon$. Then we can apply the result stated in Remark \ref{coroexpkomuro} to conclude that $\dist_\phi(x,y)<\beta$
arriving to a contradiction. 

\par The previous argument also shows that if $\phi_t(x),\phi_{h(t)}(y)\in D_p$ 
for some $t\in\R$ and $p\in\sing$ of negative index
then $\phi_t(x)$ and $\phi_{h(t)}(y)$ belongs to the same hyperbolic sector in $D_p$.
In particular $x$ and $y$ are not in separatrices. Moreover $\phi_t(x)$ and $\phi_{h(t)}(y)$
can be connected with a local cross section for all $t\in\R$.

\par Now consider an increasing divergent sequence $\{t_n\}_{n\in\N}$, $t_0=0$, such
that for all $n\in\N$ there exist $i(n)$ such that
$$\phi_{[t_n,t_{n+1}]}(x)\cup \phi_{[h(t_n),h(t_{n+1})]}(y)\subset V_{i(n)}\in\mathcal C,$$ 
where $V_{i(n)}$ may be a flow box or an adapted neighborhood $D_p$. Consider in $V_{i(n)}\cap V_{i(n+1)}$ a
local cross section $l_n$ that connects $x_n=\phi_{t_n}(x)$ and
$y_n=\phi_{h(t_n)}(y)$. Now take any
point $z\in l_0=[x,y]\subset l$. Suppose that $z$ do not belong to a separatrix of a singularity of index 0. 
Consider a sequence $s_n\in\R$ such that
$z_n=\phi_{s_n}(z)\in l_n$ for all $n\geq 0$. It exist because if
$x_n$ and $y_n$ are in $D_p$ both points are in the same hyperbolic sector. If $z_n$ is convergent then
$x_n$ and $y_n$ should be convergent too. But this is a contradiction because $x$ and $y$ are not in separatrices.
Then $z_n$ is not convergent and $z$ do not belong to any separatrix. 
\end{proof}

\begin{teo}\label{expsinsing2}
 Suppose that $\psi$ removes a singular point of $\phi$. Then $\psi$ is expansive if and only if $\phi$ is expansive.
\end{teo}

\begin{proof}
By Proposition \ref{expsinsing} it only rest to show the converse. 
By contradiction suppose that $\psi$ is not expansive and take
$\beta>0$, $x_n, y_n\in S$, $h_n\in\reparam$ and $\delta_n\to 0$ such that
$\dist(\psi_t(x_n),\psi_{h_n(t)}(y_n))<\delta_n$ for all $t\in\R$
and $\dist_\psi(x_n,y_n)>\beta$. We study the possible cases.
\par Case 1: $x_n,y_n\notin \psi_\R(p)$ for infinite values of
$n$. It contradicts the expansiveness of $\phi$ as was explained in case (1) of Proposition \ref{expsinsing}.
\par Case 2: $x_n\in \psi_\R(p)$ and $y_n\notin \psi_\R(p)$ for infinite values of $n$.
Let $t_n\in\R$ be such that $\psi_{t_n}(x_n)=p$ and define $z_n=\psi_{h_n(t_n)}(y_n)$.
Now consider $s=t-t_n$ and $g_n\in\reparam$ defined as $g_n(s)=h_n(s+t_n)-h_n(t_n)$. It is easy to check that
$\dist(\psi_s(p),\psi_{g_n(s)}(z_n))<\delta_n$ for all $s\in\R$. Taking $s=0$ we see that $z_n\to p$ since $\delta_n\to 0$. 
Therefore we can assume that $z_n\in U$ for all $n$ and it is easy to see that, eventually taking a subsequence of $z_n$,
there exist $\beta'>0$ such that
 $\dist_\phi(z_{n_1},z_{n_2})>\beta'$ if $n_1\neq n_2$. 
So $\dist(\psi_{g_{n_1}(s)}(z_{n_1}),\psi_{g_{n_2}(s)}(z_{n_2}))<\delta_{n_1}+\delta_{n_2}$ for all $s\in\R$.
Now it contradicts the expansiveness of $\phi$ because $z_{n_1}$ and $z_{n_2}$ are not in $\psi_\R(p)$.
\par Case 3: $x_n,y_n\in \psi_\R(p)$ for infinite values of $n$. 
We can suppose that $x_n=p$ for all $n$. Let $l$ be a transversal section of $\psi$ through $p$. 
Without loss of generality we can assume that $y_n\in l$ for all $n$. 
In order to apply Lemma \ref{separatricesdensas} we must notice that by Proposition \ref{singularities} the singularities of $\psi$
are of saddle type. 
Then we have that if $n$ is big enough, in the subsegment
$l'\subset l$ limited by $y_n$ and $p$
there are no separatrices of no removable singularities. 
Let $\psi'$ be a flow that removes every removable singular point
of $\psi$.
Then we have that there are not separatrices of $\psi'$ in $l'$. 
Since $\phi$ is expansive we have that: 1) by Proposition \ref{wanderings}, $\Omega(\phi)=S$ and so
$\Omega(\psi')=S$; 2) by Proposition \ref{expnoper} $\phi$ do not have periodic points and then
$\psi'$ has the same property and 3) the singular set of $\psi'$ is finite. 
Therefore we can apply Lemma \ref{irrationaltorus} to conclude that $\psi'$ is the suspension of an irrational rotation.
On the other hand $\phi$ is obtained from $\psi'$ adding singularities, 
and so it is easy to see that $\phi$ is not expansive.\qedhere
\end{proof}

\par Notice that the previous proof only used the fact that $S$ is a surface, instead of being a compact metric space,
just in case 3.

\begin{obs} \label{sing-boundary}
 Proposition \ref{expnoper} and Theorem \ref{expsinsing2} shows that expansive flows presents at least one 
singularity of negative index on each boundary component. 
\end{obs}

\section{Characterizations}

\par In this section $S$ denotes a compact surface and
$\phi\colon \R\times S\to S$ a continuous flow. The main result of this section is Theorem 
\ref{charexpsup}
that gives a characterization of expansive flows on surfaces. 
Theorem
\ref{expsinsing2} explains the generality lost supposing that the
flow do not has singularities of index 0.

\begin{teo}\label{charexpsup} 
Let $\phi$ be a flow without singularities of index 0
on a compact surface $S$. Then the following statements are equivalent:
\begin{enumerate}
 \item $\phi$ is expansive,
 \item $\sing$ is a finite and non empty set, $\Omega(\phi)=S$ and $\phi$ has no periodic points and
 \item the singularities are
of saddle type and the union of its separatrices is dense in the
surface.
\end{enumerate}
\end{teo}

\begin{proof}
$(1)\Rightarrow(2)$
It is a consequence of Propositions \ref{separatrices densas2}, \ref{wanderings} and
\ref{expnoper}.
\par $(2)\Rightarrow(3)$
By Lemma \ref{muchassep} we have that the singularities are of saddle type because by hypothesis there are no wandering points.
The union of the separatrices is dense because Proposition \ref{separatrices densas2}.

\par $(3)\Rightarrow(1)$
Given any $\beta>0$, by Lemma \ref{separatricesdensas} we have the constant $\delta$ given there.
Such $\delta$ is an expansive constant by Lemma \ref{separatricesdensas} because the union of the separatrices is dense
and there are no singularities of index 0.
\end{proof}

\par
 Notice that expansiveness implies (2) and (3) with removable singularities too. But if one adds a finite number of singularities
of index 0 to 
an irrational flow on the torus we get a flow that satisfies (2) and (3) but it is not expansive. In fact this is the only possible case.

\par To state a characterization without restrictions we introduce the following concept.
 Let $p\in S$ be a singular point of negative index of $\phi$. Consider $\phi'$ that removes all the singularities of
index 0 from $\phi$. A $\phi'$-separatrix of $p$ is called an \emph{extended separatrix} of $p$ for the flow $\phi$.

\begin{teo}\label{charact2}
A flow $\phi$ on a compact surface is expansive if and only if the singularities are
of saddle type and the union of the extended separatrices of the singularities of negative index 
is dense in the
surface.
\end{teo}

\begin{proof}
 Consider $\phi'$ the flow obtained from $\phi$ by removing every singularity of index 0. 
The proof follows easily from Theorem \ref{charexpsup}, Theorem \ref{expsinsing2} and the following fact: 
the union of the separatrices of $\phi'$ is equal to the union of the extended separatrices of the singularities of negative index.
\end{proof}

\par Now we are going to give a characterization of the surfaces admitting expansive flows. For this we will introduce some surgery tools.
Let $S$ and $S'$ be two compact surfaces and $\phi\colon S\times \R\to S$ and 
$\phi'\colon S'\times \R\to S'$ be two continuous flows. 
A \emph{semi-conjugacy} from $\phi'$ to $\phi$ is a surjective and continuous map $h\colon S'\to S$ such that 
$\phi_t\circ h=h\circ \phi'_t$ for all $t\in\R$. If $h$ is a homeomorphism then it is called a \emph{conjugacy}.

\par If $\gamma\subset S\setminus\partial S$ and $\gamma'_1,\gamma'_2\subset
\partial S'$ are saddle connections then we say that $\phi'$ \emph{makes a cut along} $\gamma$ if there exists
a semi-conjugacy $h\colon S'\to S$ such that 
$h\colon S'\setminus \clos(\gamma'_1\cup\gamma'_2)\to S\setminus\clos(\gamma)$ is a conjugacy from
$\phi'|_{S'\setminus \clos(\gamma'_1\cup\gamma'_2)}$ to 
$\phi|_{S\setminus\clos(\gamma)}$. In that case we also say that $\phi'$ \emph{glues} $\gamma'_1$ with $\gamma'_2$.

\begin{df}
 Given a flow $\phi$ on $S$ we say that another flow $\phi'$ on $S'$ is obtained from $\phi$ by a \emph{basic operation}
if it is obtained by adding o removing a singularity of index 0, 
by gluing saddle connections or cutting along a saddle connection.
\end{df}

\par It is easy to see that every saddle connection not contained in the boundary of the surface, can be cut.
Also, every pair of saddle connections in the boundary can be glued.

\begin{obs}\label{addboundary}
We will show a special way to add a boundary with basic operations. That construction will be useful in Theorem \ref{toposupexp}. 
Take two points $p$ and $q=\phi_t(p)$, $t>0$, in a regular orbit and put singularities of index 0 in them. 
Make a cut on the saddle connection determined by $p$ and $q$.
Let $\gamma$ and $\gamma'$ be the saddle connections in the boundary connecting $p$ and $q$. Put two singular points $r$ and $s$ 
of index 0 in $\gamma$. Now glue the saddle connections that starts in $p$. See
figure \ref{cajaconborde}. 
\end{obs}

\begin{figure}[htbp]
\input{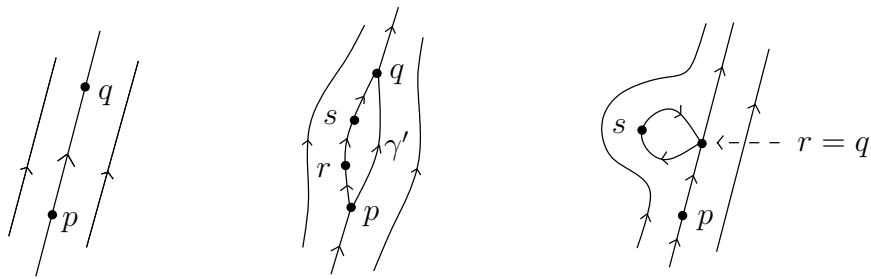}
    \caption{Adding a boundary.}
   \label{cajaconborde}
\end{figure}

Recall that every compact connected surface $S$ is obtained from the sphere attaching 
$h\geq0$ handles, $b\geq 0$ boundaries and $c\geq 0$ cross-cups. Denote by $S^{h,b,c}$ such surface.
It is well known that if $c\geq 3$ then $S^{h,b,c}$ is homeomorphic to $S^{h+1,b,c-2}$. 
So we will assume that $c=0,1,2$.
See for example \cite{Hirsch}

\begin{teo}\label{toposupexp}
A compact connected surface $S$ admits an expansive flow if and only if $h>0$ and $h+b+c>1$.
\end{teo}

\begin{proof}
\par (\emph{Direct.}) 
Suppose that $S$ admits an expansive flow. By Theorem \ref{expsinsing2} we can assume that this flow do not have removable singularities.
So the index of each singular point is negative. Then the Euler's characteristic of $\mathcal{X}(S)$ is negative.
Theorem \ref{charexpsup} implies that expansive flows has recurrent non-trivial orbits. 
If $c=0$, i.e. $S$ is orientable, then by the results in \cite{Maier} (see Theorem 2 in \cite{Markley}) 
we have that the genus of $S$ is positive
and then $h>0$. In the orientable case $\mathcal{X}(S)=2-2h-b$ and then $\mathcal{X}(S)<0$ implies $2<2h+b$. Therefore
$h+b>1$.
If $c>0$, the non-orientable case, we can apply Theorem 4 of \cite{Markley} to conclude that the genus of $S$ is greater than 2.
In the non-orientable case the genus of $S$ equals $c+2h$, then $c+2h>2$ and since $c\leq 2$ we have that $h>0$. 
Now $c>0$ and $h>0$ implies $c+h>1$.
\par (\emph{Converse.}) The conditions $h>0$ and $h+b+c>0$ are equivalent with saying that $S$ is a torus with
$b$ boundaries, $h'$ handles and $c'$ cross-cups attached such that $h'+b+c'>0$. So we will show that such surfaces admits expansive flows.
Given $k>$ we consider an expansive flow on the torus with $k$ boundaries. 
It can be done in the following way: take an irrational flow
on the torus, add $k$ disjoint boundaries as in Remark \ref{addboundary}.
In this way we get an expansive flow $\phi$ on $S=S^{1,k,0}$ such that if $\gamma$ is a component of $\partial S$ 
then $\sing\cap\gamma=\{p,q\}$ and there exist two regular points $a,b\in\gamma$ such that $\omega(a)=\alpha(b)=\{q\}$
and $\omega(b)=\alpha(a)=\{p\}$. 
\par Given $h',b,c'\geq 0$ with $h'+b+c'>0$ consider $k= 2h'+c'+b$. So $\partial S=\cup_{i=1}^{i=k}\gamma_i$ being 
each $\gamma_i$ homeomorphic to a circle. It is easy to see that two boundaries can be glued to obtain a handle without
loosing expansiveness. Also each boundary can make a cross-cup in the following way. Take $a,b,p,q\in \gamma_i$ as explained above.
Identify $p$ with $q$ and $\phi_t(a)$ with $\phi_t(b)$ for all $t\in\R$. 

\end{proof}

\begin{obs}
 The compact connected surfaces that do not admit expansive flows are: the torus, the sphere with $b$ boundaries, the projective plane with $b$ boundaries and 
the Klein's bottle with $b$ boundaries (with $0\leq b<\infty$ in the three cases).
\end{obs}

\begin{teo}\label{charact3}
 Suppose that $S$ is a compact surface different of the torus. A flow $\phi$ on $S$ is expansive if and only if $\Omega(\phi)=S$,
$\phi$ do not has periodic orbits and $\sing$ is finite.
\end{teo}

\begin{proof}
 \par \emph{Direct.} It is a consequence of Propositions \ref{wanderings} and
\ref{expnoper} and the fact that expansive flow on compact spaces has a finite number of singularities.
\par\emph{Converse.} Consider a flow $\phi'$ removing the singularities of index 0 from $\phi$. By Theorem \ref{expsinsing2} we just have to prove
that $\phi'$ is expansive. Notice that $\Omega(\phi)=\Omega(\phi')$, $\sing(\phi')$ is finite and $\phi'$ do not have periodic orbits. 
Since $S$ is not the torus we can apply Lemma \ref{irrationaltorus} to conclude that the union of the separatrices of $\phi'$ is dense
in the surface. Now since $\phi'$ has not singularities of index 0 we have that it is expansive.
\end{proof}

\begin{cor}
 If $\phi$ is an expansive flow on $S$ and $\phi'$ is a flow on $S'$ obtained from $\phi$ by a basic operation then $\phi'$ is expansive if and only if $S'$ is not the torus.
\end{cor}

\begin{proof}
\par\emph{Direct}. By Theorem \ref{toposupexp}, $S'$ can not be the torus if $\phi'$ is expansive.
\emph{Converse}. By Theorem \ref{expsinsing2} expansiveness is invariant under adding or removing singularities. 
It is easy to see that the non-wandering set do no change if one cuts or glue saddle connections. 
Also, periodic orbits can not be created with basic operations if $\phi$ is expansive and the singular set of $\phi'$ is finite.
So we conclude by Theorem \ref{charact3}.
\end{proof}

\section{Interval exchange maps}

In this section we will introduce the definition of \emph{expansive interval exchange map} and show its relation with the expansiveness 
of it \emph{suspension} flow.

Let $S^1=\R/\Z$ be the circle and consider
$A,B\subset S^1$ finite sets, we say that $f\colon
S^1\setminus A\to S^1\setminus B$ is an \emph{interval exchange map}
if it is an homeomorphism that preserves the Lebesgue measure of $S^1$. 
A point $a\in A$ is said to be \emph{singular} if $\lim_{x\to a^-}f(x)\neq \lim_{x\to a^+}f(x)$
and let $\sing_f$ be the set of singular points of $f$.
Denote by $\sing_f^*$ the set of points $x\in S^1$ such that there exists $n\geq 0$ with 
$f^n(x)\in \sing_f$. 
Defining $f^0(a)=a$ for a point $a\in A$ we have that $A\subset \sing_f^*$. 

\begin{df}
 We say that an interval exchange map is \emph{expansive} if 
there exist $\delta>0$ such that if $x, y\notin \sing_f^*$, $x\neq y$, then there exist $n\geq 0$ such that
$\dist(f^n(x),f^n(y))>\delta$.
\end{df}

\par Now consider a compact surface $S$. An embedded circle $\gamma\subset S$ 
is said to be a \emph{quasi-global cross section} for a 
flow $\phi$ if it is transversal to $\phi$
and intersects every regular orbit. If a flow $\phi$ on $S$ with a finite number of singularities has a quasi-global cross section 
such that its first return map is conjugated to an interval exchange map $f$
then $\phi$ it is said to be a \emph{suspension} of $f$. 
It is easy to see that the surface $S$ can not have boundary, that is because $\gamma$
intersects every non-singular orbit.

\begin{obs}
Consider a suspension $\phi$ of an interval exchange map $f$. Then $A=\sing_f$ if and only if $\phi$ has no singular point of index 0.
For simplicity we will assume that every point of $A$ is singular, so $A=\sing_f$. 
Consequently the suspensions considered will not have singularities of index 0.
\end{obs}

In \cite{CoboGutierrez} (Lemma 8) it is shown how interval exchange maps can be suspended.

\begin{teo}
Let $f$ be an interval exchange map and $\phi$ a suspension of $f$. 
The following statements are equivalent.
\begin{enumerate}
\item $\phi$ is expansive.
\item $f$ is expansive.
\item $\sing^*_f$ is dense in $S^1$.
\item $f$ has no periodic orbits and $\sing_f\neq\emptyset$.
\end{enumerate}
\end{teo}

\begin{proof}
\par (1 $\to$ 3) It follows by Theorem \ref{charexpsup} item (3).
\par (3 $\to$ 2) Let $I_1,\dots,I_n$ be the intervals exchanged by $f$. We will show that
if $\delta>0$ and $3\delta$ is less than the length of each $I_i$ 
then $\delta$ is an expansive constant for $f$. 
If $\dist(x,y)<\delta$ we denote by $(x,y)$ the smallest interval determined by $x$ and $y$.
In this case there is at most one singular point in $(x,y)$.
Now fix $x,y\notin \sing^*$ such that $\dist(x,y)<\delta$.
By hypothesis there exists a smallest $n\geq 0$ and a point $z'\in(x,y)$ such that 
$z=f^n(z')\in \sing_f$.
Let $f(z^\pm)=\lim_{u\to z^\pm} f(u)$. Without loss of generality suppose that $\dist(f^{n+1}(x),f(z^-)),\dist(f^{n+1}(y),f(z^+))<\delta$.
Since $z\in\sing_f$ we have that $\dist(f(z^-),f(z^+))\geq 3\delta$. Then it follows that $\dist(f^{n+1}(x),f^{n+1}(y))\geq\delta$.

\par (2 $\to$ 4) By contradiction suppose that $x$ is a periodic point of $f$. It is easy to see that there exists a neighborhood $U$ of $x$ 
such that $U\cap \sing_f^*=\emptyset$, moreover every point in $U$ is periodic. This easily contradicts the expansiveness of $f$. 
If $\sing_f=\emptyset$ then $f$ is an homeomorphism of the circle since we are assuming that $A=\sing_f$. 
This gives a contradiction too, because there are no expansive homeomorphisms on the circle, as proved in \cite{JU}.
\par (4 $\to$ 1) It follows by Theorem \ref{charexpsup} item (2).
\end{proof}

\par Now we shall study \emph{quasi-minimal} flows.

\begin{df}
A flow is said to be \emph{quasi-minimal} if there exist a finite set $X\subset S$ such that the orbit of each 
$x\in S\setminus X$ is dense. 
\end{df}

\begin{obs}
It is easy to see that $\phi$ is quasi-minimal if and only if every regular orbit is dense and $\sing$ is a finite set.
\end{obs}

A separatrix $\gamma$ whose $\omega$-limit and
$\alpha$-limit sets are singular points (may be the same)
is called a
\emph{saddle connection}.

\begin{lemma}\label{recdoslados} 
Suppose that $\Omega(\phi)=S$, $x\in S$ is not periodic and $\omega(x)$ is not a singular point.
Then if $l$ is an open transversal section with an extreme point
$x$ there exist $t>0$ such that $\phi_t(x)\in l$.
\end{lemma}

\begin{proof}
Let $l^*=\{y\in l:\phi_{\R^+}(y)\cap l\neq\emptyset\}$ be the set of points of $l$ returning to $l$
and consider the first return map $f\colon
l^*\to l$. 
Suppose that $z$ is the extreme point of $l$ different of $x$.
Notice that if $f(y)\neq x$  and $f(y)\neq z$ then $y$ is an interior point of $l^*$.
Therefore $l^*$ has at least two non-interior points.
Let $(a,b)\subset l^*$ be a connected component of $l^*$. If
$a\notin l^*$ then by Lemma \ref{peixoto2}, we have that
$\omega(a)$ is a singular point and then $a\neq x$. 
Since $\Omega(\phi)=S$ we have by Lemma \ref{muchassep} that there is just a finite number of separatrices.
So there is just a finite number of points in $l\setminus l^*$ whose $\omega$-limit set is a singularity.
On the other hand if
$a\in l^*$ then either $f(a)$ is an extreme point of $l$ or $a$ is an extreme point of. 
The same consideration can be made for $b$.
So $l^*$ has finite number of connected components.
Also, since $\Omega(\phi)=S$, we have that $l^*$ is dense in $l$.
Then $x$ belongs to the closure of some connected component of $l^*$.
Now we can apply Lemma 
\ref{peixoto2} to conclude that the positive orbit of $x$ meets the closure of $l$.
Since $x$ is not periodic the only possible problem is that $x$ returns to $z$, the other extreme of $l$.
It can be solved considering a subsegment of $l$ from the beginning of the proof.
\end{proof}

\begin{prop}\label{minimal-iem}
The following statements are equivalent.
\begin{enumerate}
 \item $\phi$ is quasi-minimal.
 \item $\phi$ is a suspension of a minimal interval exchange map $f$.
 \item $\sing$ is finite, $\Omega(\phi)=S$, $\phi$ has neither periodic orbits nor saddle connections.
\end{enumerate}
\end{prop}

\begin{proof}

\par($1\to 2$) It is easy to see that quasi-minimal flows are Cherry flows. Moreover, there exist $x\in S$ such that $\phi_{\R^+}(x)$
is dense. By \cite{Gutierre} (section 4) we have that $\phi$ is a suspension of a minimal interval exchange map. 

\par($2\to 3$) Since $f$ preserves the Lebesgue measure, $\Omega(\phi)=S$. The flow $\phi$ can not have periodic orbits because
$f$ is minimal. The singular set of $\phi$ is finite by definition of suspension. In \cite{Keane} it is shown that every orbit of
a minimal interval exchange map is infinite, then $\phi$ do not have saddle connections.

\par($3\to 1$) Suppose there exist a regular point $x\in S$ such that $\omega(x)$ is not the whole $S$. 
Since $S$ is connected we have that
$\partial\omega(x)\neq\emptyset$. Therefore
there exist a regular point $y\in\partial\omega(x)$. 
Take a flow box $U$ around $y$. 
Since $U\setminus \omega(x)\neq \emptyset$ we can suppose
that $y$ belongs to the frontier of a connected component of
$U\setminus \omega(x)$. Since $\omega(x)$ and its complement
on $S$ are invariant sets by the flow, we have that $\omega(x)$ is a singular point. If this were not the cases
we can apply Lemma
\ref{recdoslados} concluding that $y$ returns to the complement of
$\omega(x)$ which is an absurd. 

\par So we have shown that if $x$ is a regular point and $\omega(x) \neq S$ then $\omega(x)$ is a singular point. 
Arguing the same for $\alpha(x)$ and using the fact that there are no saddle connections, we have that the orbit of every regular
point is dense.
\end{proof}

\begin{cor}\label{quasiminimal}
Expansive flows on connected surfaces without saddle connections
are quasi-minimal.
\end{cor}

\begin{proof}
It follows by Theorem \ref{charact3} and Proposition \ref{minimal-iem}.
\end{proof}

\section{Global Structure of Expansive Flows Of Surfaces}

\par A union of saddle connections can separate the dynamics of a flow on a surface. 
So if one cuts every saddle connection the surface may be disconnected.
In this section we will study this decomposition for expansive flows, obtaining irreducible sub-dynamics. 
Similar structure theorems are given in \cites{Maier,CoboGutierrez}.
First we give an example showing the main ideas of Theorem \ref{structure}.
 
\begin{ejp} \label{bitoroexp}
Consider an irrational flow on the torus. 
Take any orbit segment with endpoints $p$ and $q$. Put singularities 
in $p$ and $q$ and make a cut along this saddle connection. 
Now take a copy and glue the saddle connections in the boundary getting a bi-torus.
Let us show that the flow is expansive. There are two singular points and each one has index -1.
And since the orbits of the irrational flow on the torus are dense, we have that the union of the separatrices is dense in the bi-torus.
Now applying Theorem
\ref{charexpsup} we conclude that the flow is expansive. 
That example shows how the saddle connections separates the surface.
\end{ejp}

\begin{obs}
 In the proof of Theorem \ref{structure} 
we will need to \emph{collapse} a boundary component of the surface. Now we will show
that this can be done using just basic operations. 
\par First remove every singularity of index 0. Take a singular point $p$ 
in a component of the boundary $\gamma$ 
(by Remark \ref{sing-boundary} each boundary component contains at least one singular point of negative index). 
Consider a separatrix $\gamma$ of $p$ such that $\gamma$ with another separatrix of $p$ in the 
boundary they determine a hyperbolic sector of $p$. Put a singularity $q$ of index 0 in $s$. Make a cut along the saddle connection 
with extreme points $p$ and $q$. In this way $p$ splits in $r$ and $s$. Now either $r$ or $s$ is
of index 0. Suppose that it is $r$, as in figure \ref{colapsoborde}, and remove it. Now glue the saddle connections
that connects $q$ and $s$. Notice that $q$ can be removed. 
\par Repeating this procedure we remove the boundary components. 
The figure \ref{colapsoborde} shows this procedure in the special case where there is only one saddle connection 
in a boundary component, but it is work in any case. Also, in the figure we see that finally $s$ can be removed. This fact depends on 
how many interior separatrices had $p$ in the beginning.

\begin{figure}[h]
\input{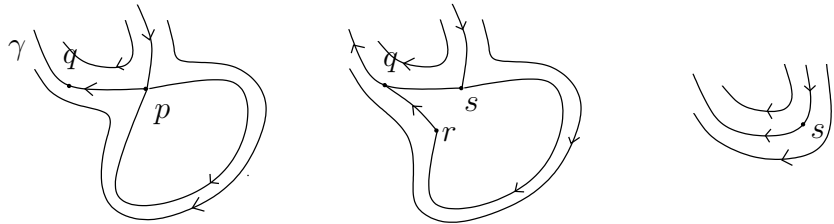}
\caption{Removing a saddle connection in the boundary.}
\label{colapsoborde}
\end{figure}

\end{obs}

\begin{teo}\label{structure}
 Every expansive flow on a compact surface can be obtained in the following way:
\begin{enumerate}
 \item Take $f_1,\dots,f_n$ minimal interval exchange maps.
\item Consider their suspensions $\phi_1,\dots,\phi_n$ on the surfaces\\$S_1,\dots,S_n$.
\item Do a finite number of basic operations on $\cup_{i=1}^n S_i$.
\end{enumerate}
\end{teo}

\begin{proof} Take each interior saddle connection and make a cut on the surface through this orbit.
What we get is a surface with boundary that may not be connected.
Let $S_1,\dots,S_N$ be such connected components and let $\phi_i$ be the flow on $S_i$
induced by $\phi$. By construction the saddle connections are in the boundary of $S_i$.
Now collapsing each boundary component to a singular point we get flows without wandering points, without periodic orbits
and without saddle connections. This are suspensions of minimal interval exchange maps by Proposition \ref{minimal-iem}.
\end{proof}

\par  Conversely, if the procedure of Theorem \ref{structure} gives you a connected surface different from the torus, 
then the flow is expansive by Theorem \ref{charact3}.

\section{Rational Billiards}

\par Consider a polygon $P\subset \R^2$ with angles $\theta_i\in\Q\pi$, $i=1,\dots, n$. In the literature it is called
a \emph{rational polygon}.
Let $V_1,...,V_n$ be its corners and denote by $l_1,...,l_n$ its sides. 
Consider $n$ lines $r_1,...,r_n$, $0\in r_i$, parallel to each $l_i$ respectively and 
let $S_1,...,S_n$ be the reflections associated with the lines $r_i$. 
Let $G$ be the group of isometries generated by $\{S_i\}$. 
On the set $P\times G$ we define an equivalence relation generated by the following
property: if $x\in l_i\subset\partial P$ then $(x,g)\simeq (x,S_ig)$. 
Let $S$ be the quotient space $P\times G/\simeq$. Denote by $p_i$ the class of the corner $V_i$ for all $i=1,\dots,n$
and define $\sing=\{p_1,\dots,p_n\}$. 
Endowed with the quotient topology it is known that $S$ is a connected closed surface (see \cite{ZK} Proposition 1).
Also, the planar metric of $P$ induces a flat metric in $S\setminus \{V_1,\dots,V_n\}$ and 
the parallel transport do not depends on the curve (see \cite{ZK} Proposition 2). 

\par With this properties we can construct a flow on $S$. 
Fix a direction $v\in\R^2$, $||v||=1$. Define a vector field $Y$ on $S\setminus \sing$, with the parallel transport of $v$.
Consider a non negative smooth function $\rho\colon S\to \R$ 
vanishing only in the singular points.
Define the vector field $X$ in the whole surface $S$ 
as $X(x)=\rho(x)Y(x)$. 
Let $\phi_v$ be the flow on $S$ associated with the vector field $X$.

\begin{teo}\label{billar}
The associated flow $\phi_v$ is expansive    
if and only if $S$ is not the torus and 
there are no periodic orbit in the polygon with initial direction $v$. 
\end{teo}

\begin{proof}
\par We have that $\Omega(\phi_v)=S$ because there is an invariant measure that is positive on open sets (see \cite{ZK}). 
Hence we can apply 
Theorem \ref{charact3} to conclude.
\end{proof}

\par It is known that the surface $S$ is the torus if and only if
the polygon has all of its angles of the form $\pi/n$ with $n=2,3,4,...$. 
This polygons are: rectangles and triangles $(\frac \pi3,\frac\pi 3,\frac\pi 3)$, 
$(\frac\pi 2,\frac\pi 4,\frac \pi 4)$ and $(\frac\pi 2,\frac\pi 3,\frac\pi 6)$.

\par 
The previous result gives us a family of examples of expansive flows.
Fix a polygon $P$ and notice that the set of direction $v$
such that there exist a periodic orbit with initial direction $v$, is countable.
Therefore, in order to get an expansive flow, just take $v$ without periodic orbits.


\begin{bibdiv}
\begin{biblist}

\bib{BW}{article}{
author={R. Bowen and P. Walters}, title={Expansive One-Parameter
Flows}, journal={J. Diff. Eq.}, year={1972}, pages={180--193},
volume={12}}

\bib{CoboGutierrez}{article}{
    AUTHOR = {M. Cobo},
    AUTHOR = {C. Gutierrez},
    AUTHOR = {J. Llibre},
     TITLE = {Flows without wandering points on compact connected surfaces},
   JOURNAL = {Trans. Amer. Math. Soc.},
      VOLUME = {362},
      YEAR = {2010},
    NUMBER = {9},
     PAGES = {4569--4580},}


\bib{GKT}{article}{
author={G. Galperin},
author={T. Kruger},
author={S. Troubetzkoy},
title={Local instability of orbits in polygonal and polyhedral billiards}, journal={Comm. Math. Phys.}, year={1995}, pages={463-473},
volume={169}}

\bib{Gutierre}{article}{
author={C. Gutierrez},
title={Smoothability of Cherry flows on two-manifolds},
journal={Lecture Notes in Math.},
volume={1007},
year={1983},
pages={308--331}}


\bib{Gutierrez}{article}{
author={C. Gutierrez},
title={Smoothing continuous flows on two-manifolds and recurrences},
journal={Ergod. Th \& Dynam. Sys.},
volume={6},
year={1986},
pages={17--14}}

\bib{Hartman}{book}{
author={P. Hartman},
title={Ordinary Differential Equations},
publisher={John Wiley \& Sons},
year={1964}}

\bib{H}{article}{
author={K. Hiraide},
title={Expansive Homeomorphisms of Compact Surfaces are Pseudo Anosov},
journal={Osaka Journal of Math.},
volume={27},
number={1},
year={1990},
pages={117--162}}

\bib{Hirsch}{book}{
author={M. W. Hirsch},
title={Differential Topology},
publisher={Springer-Verlag},
year={1976}}

\bib{JU}{article}{
author={J. F. Jakobsen and W. R. Utz}, title={The non-existence of
expansive homeomorphisms on a closed $2$-cell}, journal={Pacific
J. Math.}, year={1960}, volume={10}, number={4},
pages={1319--1321}}

\bib{Keane}{article}{
author={M. Keane}, title={Interval exchange transformations}, 
journal={Math. Z.}, 
year={1975}, 
volume={141}, 
pages={25--31}}

\bib{K}{article}{
author={M. Komuro}, title={Expansive properties of Lorenz
attractors}, journal={The Theory of dynamical systems and its
applications to nonlinear problems}, year={1984}, place={Kyoto},
pages={4--26}, publisher={World Sci. Singapure}}

\bib{L}{article}{
author={J. Lewowicz},
title={Expansive homeomorphisms of surfaces},
journal={Bol. Soc. Bras. Mat.},
year={1989},
volume={20},
number={1},
pages={113--133}}

\bib{LG}{article}{author={He Lianfa and Shan Guozhuo}, title={The Nonexistence of Expansive Flow on a Compact 2-Manifold}, journal={Chinese Annals of Mathematics},
volume={12}, number={2}, pages={213-218}, year={1991} }

\bib{Maier}{article}{
 title = {Trajectories on orientable surfaces},
     author = {A. G. Maier},
     journal = {Sb. Math.},
     volume = {12},
     number = {1},
     pages = {71--84},
     year = {1943},}


\bib{Markley}{article}{
 title = {On the Number of Recurrent Orbit Closures},
     author = {N. G. Markley},
     journal = {Proceedings of the American Mathematical Society},
     volume = {25},
     number = {2},
     pages = {413--416},
     year = {1970},
     publisher = {American Mathematical Society},}



\bib{Oka}{article}{
author={M. Oka}, title={Expansiveness of real flows},
journal={Tsukuba J. Math}, year={1990}, volume={14}, number={1},
pages={1--8} }

\bib{W}{article}{
author={H. Whitney},
title={Regular Family of Curves},
journal={Annals of Mathematics},
year={1933},
volume={34},
number={2},
pages={244--270}}

\bib{ZK}{article}{
author={A. N. Zemlyakov and A. B. Katok}, title={Topological
transitivity of billiards in polygons}, journal={Mat. Zametki},
volume={18}, number={2}, pages={291--300}, year={1975}}

\end{biblist}
\end{bibdiv}
\end{document}